\documentclass[12pt,leqno]{amsart}

\usepackage[a4paper,margin=25mm]{geometry}
\usepackage{graphicx}
\usepackage{float}
\usepackage{tikz}

\newtheorem{theorem}{Theorem}
\newtheorem{proposition}[theorem]{Proposition}
\newtheorem{lemma}[theorem]{Lemma}
\newtheorem{corollary}[theorem]{Corollary}

\newcommand{\vol}{\operatorname{vol}}
\newcommand{\conv}{\operatorname{conv}}

\title{Lattice coverings by parallelotopes containing cube vertices}
\author{Dragomir Grozev}
\address{D. Grozev\\Institute of Mathematics and Informatics\\Bulgarian Academy
of Sciences\\Acad. G. Bonchev 8, 1113 Sofia, Bulgaria}
\email{drago.grozev@gmail.com}

\author{Nikolai Nikolov}
\address{N. Nikolov\\Institute of Mathematics and Informatics\\Bulgarian Academy
of Sciences\\Acad. G. Bonchev 8, 1113 Sofia, Bulgaria
\vspace{1mm}
\newline Faculty of Information Sciences\\
State University of Library Studies and Information Technologies\\
Shipchenski prohod 69A, 1574 Sofia, Bulgaria}
\email{nik@math.bas.bg}
\thanks{N. Nikolov was partially supported by the Bulgarian National Science Fund,
Ministry of Education and Science of Bulgaria under contract KP-06-N82/6.}


\begin{document}

\begin{abstract}
We investigate when a parallelotope containing vertices of the unit cube
$Q_n=[0,1]^n$ covers $\Bbb R^n$ by integer translates. We prove the
covering property for every parallelepiped containing five vertices of $Q_3$,
and, in any dimension, for every parallelotope containing a facet of $Q_n$ and
one vertex of the opposite facet. As a consequence, we obtain the sharp bound
$1/n$ for the volume of a pyramid contained in $Q_n$ with a parallelotope base.
We also give an alternative proof of this bound.
\end{abstract}

\maketitle

\section{Results}
Let $Q_n=[0,1]^n$, where $n\ge 2$. We consider the following question.
\smallskip

\noindent{\bf Question.} Does every parallelotope $P\subset\Bbb R^n$
containing more than half of the vertices of $Q_n$ satisfy
$P+\Bbb Z^n=\Bbb R^n$?
\smallskip

Obviously, the threshold cannot be lowered.

For $n=2$, the answer is affirmative. More generally, a centrally symmetric
convex body in $\Bbb R^2$ containing three noncollinear lattice points
covers the plane by integer translates (see, e.g., \cite{Chalk}).
We prove that the answer is also affirmative for $n=3$.

\begin{proposition}\label{prop:1}
Let $P\subset\Bbb R^3$ be a parallelepiped containing five vertices of $Q_3$.
Then $P+\Bbb Z^3=\Bbb R^3$.
\end{proposition}

Our main result is the following partial answer to the posed question.

\begin{theorem}\label{thm:2}
Let $P\subset\Bbb R^n$ be a parallelotope such that
$Q_{n-1}\times\{0\}\subset P$ and $e_n\in P$.
Then $P+\Bbb Z^n=\Bbb R^n$.
\end{theorem}

As a corollary, the standard volume bound for lattice coverings gives
$\vol(P)\ge1$. Then an affine change of coordinates implies
the following consequence.

\begin{corollary}\label{cor:3}
A pyramid in $Q_n$ with base an $(n-1)$-dimensional parallelotope has a volume
at most $1/n$.
\end{corollary}

The constant $1/n$ is sharp as the pyramid with base $[0,1]^{n-1}\times\{0\}$ and 
apex $p=e_n$ shows. We also give an alternative proof of Corollary~\ref{cor:3}.

\section{Proofs}
The covering proofs use the following result.

\begin{lemma}
\label{lem:4} \cite[Lemma~5.3]{B}
Let $K\subset \Bbb R^n$ be a convex body and $\mathcal{T}$ be a triangulation of 
$\Bbb R^n$ with the diameters of simplices uniformly bounded from above. Let $\Lambda$ 
be the vertex set of $\mathcal T$. Assume that for every simplex $\sigma\in\mathcal T$, 
there exists $t_\sigma\in\Bbb R^n$ with $\sigma\subset K+t_\sigma$. 
Then $(K+t)\cap\Lambda\neq\emptyset$ for every $t\in\Bbb R^n$.
\end{lemma}

We apply the lemma with $\Lambda=\Bbb Z^n$ and $K=P$, where $P$ is a
parallelotope. We triangulate $Q_n$ in a way that extends periodically to
$\Bbb R^n$. The key step is to show that each simplex of the triangulation
can be translated into $P$.

The next lemma characterizes when a simplex can be translated into a
parallelotope.

\begin{lemma}
\label{lem:5}
A simplex $\sigma$ can be translated into a parallelotope $P$ if and only if
for every pair $u,v$ of vertices of $\sigma$ we have $u-v\in P-P$.
\end{lemma}

\begin{proof}
Let
$$P=\left\{x\in \Bbb R^n :
a_i\le f_i(x)\le a_i+\Delta_i,\quad 1\le i\le n \right\},$$
where the $f_i$ are linear functionals (corresponding to the opposite facets) and
$a_i,\Delta_i \in \Bbb R$, $\Delta_i>0, 1\le i\le n$. It follows that
$$P-P= \left\{x\in \Bbb R^n: |f_i(x)|\le \Delta_i,\quad 1\le i\le n \right\}.$$

Let $\sigma_0,\ldots,\sigma_n$ be the vertices of $\sigma$. Suppose that
$\sigma_i-\sigma_j\in P-P$ for every $i,j$, and set
$$m_i:=\min_j f_i(\sigma_j),\quad M_i:= \max_j f_i(\sigma_j),\quad 1\le i\le n.$$
Thus, $M_i-m_i\le \Delta_i$. Hence, for every $\sigma_j, 0\le j\le n$, we have
$$m_i\le f_i(\sigma_j)\le M_i\le m_i+\Delta_i.$$
Since $P$ is nondegenerate, the functionals $f_i$ are linearly independent, so
there exists $t\in\Bbb R^n$ with $f_i(t)=m_i-a_i$. It follows that
$\sigma\subset t+P$. The converse is immediate.
\end{proof}

\begin{proof}[Proof of Theorem~\ref{thm:2}]
Reflecting the first $n-1$ coordinates, if necessary, we may assume that
$\mathbf{1}:=e_1+\cdots+e_n\in P$ in place of $e_n\in P$. These reflections
preserve both $Q_{n-1}\times\{0\}$ and the integer lattice. For any
permutation $\tau$ of
$\{1,2,\dots,n\}$, consider the simplex
$$T_{\tau}:= \conv\{v_0,v_1,\ldots,v_n\};
\quad v_0:=0,\quad v_i:=\sum_{j=1}^i e_{\tau(j)},\ 1\le i\le n.$$
Note that $\mathcal{T}:=\{T_{\tau}: \tau\in S_n\}$ is a triangulation of $Q_n$
that can be extended by translations to a triangulation of $\Bbb R^n$ with
vertex set $\Bbb Z^n$. For any two vertices $u,v$ of $T_{\tau}$, we have
$$u-v=\pm \sum_{j\in A} e_j,\quad \text{ where } A\subset \{1,2,\ldots,n\}.$$
If $n\notin A$, the sum is a vertex of $Q_{n-1}\times\{0\}$; if $n\in A$,
it is the difference between $\mathbf{1}$ and another vertex of that facet.
Since $0,\mathbf{1}\in P$, in both cases $u-v\in P-P$. Hence
Lemma~\ref{lem:5} shows that each simplex can be translated into
$P$. The periodic triangulation has uniformly bounded simplex diameters, so
Lemma~\ref{lem:4} gives $P+\Bbb Z^n=\Bbb R^n$.
\end{proof}

\begin{proof}[Proof of Proposition~\ref{prop:1}]
Write the vertices of $Q_3$ as $ijk$, where $i,j,k\in\{0,1\}$, and set
$A=000$, $B=100$, $C=110$, $D=010$, $A_1=001$,
$B_1=101$, $C_1=111$, and $D_1=011$.
Up to symmetries of the cube, there are three configurations of five vertices:

\noindent(i) $A,B,C,D,A_1$;\quad (ii) $A,B,D,A_1,C_1$;\quad (iii) $A,B,C,D_1,C_1$.
\smallskip

(i) The covering follows from Theorem~\ref{thm:2}.

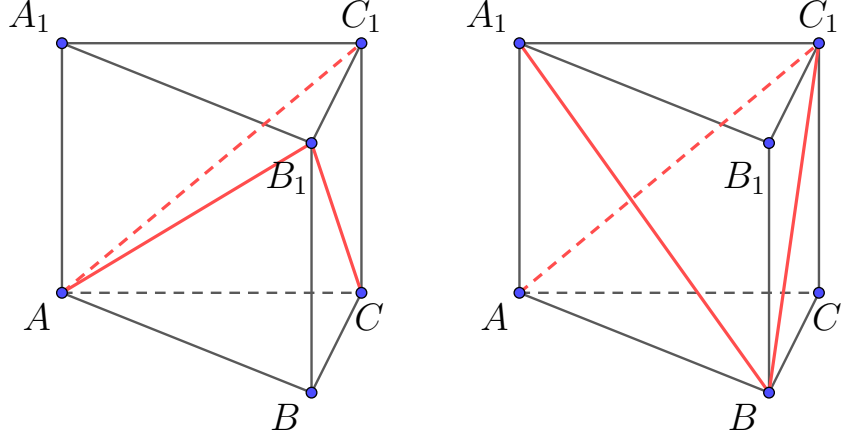
\begin{figure}[H]
\caption{Two triangulations of a triangular prism.
The added diagonals are shown in red.}
\label{fig:1}
\centering
\begin{tikzpicture}[
    scale=1.1,
    edge/.style={
        draw=black!65,
        line width=0.9pt
    },
    diagonal/.style={
        draw=red!70,
        line width=1.2pt
    },
    hidden/.style={
        dash pattern=on 4pt off 3pt
    },
    vertex/.style={
        circle,
        draw=black,
        fill=blue!70,
        line width=0.5pt,
        inner sep=0pt,
        minimum size=4pt
    },
    every node/.style={font=\large}
]

\foreach \offset/\variant in {0/1, 5.5/2}{
    \begin{scope}[xshift=\offset cm]
        \coordinate (A) at (0,0);
        \coordinate (B) at (3,-1.2);
        \coordinate (C) at (3.6,0);

        \coordinate (A1) at (0,3);
        \coordinate (B1) at (3,1.8);
        \coordinate (C1) at (3.6,3);

        \draw[edge] (A)--(B)--(C);
        \draw[edge,hidden] (A)--(C);
        \draw[edge] (A1)--(B1)--(C1)--cycle;
        \draw[edge] (A)--(A1)
                    (B)--(B1)
                    (C)--(C1);

        \draw[diagonal,hidden] (A)--(C1);
        \ifnum\variant=1\relax
            \draw[diagonal] (A)--(B1)--(C);
        \else
            \draw[diagonal] (A1)--(B)--(C1);
        \fi

        \foreach \point in {A,B,C,A1,B1,C1}{
            \node[vertex] at (\point) {};
        }

        \node[below left] at (A) {$A$};
        \node[below left] at (B) {$B$};
        \node[below, xshift=3pt] at (C) {$C$};
        \node[above left] at (A1) {$A_1$};
        \node[below left, yshift=-3pt, xshift=3pt]
            at (B1) {$B_1$};
        \node[above] at (C1) {$C_1$};
    \end{scope}
}
\end{tikzpicture}
\end{figure}
(ii) We triangulate the prism $ABCA_1B_1C_1$ using the edges
$AC_1, AB_1, B_1C$. The prism $ADCA_1D_1C_1$ is triangulated by translating $AB_1$
and $B_1C$ onto the opposite facets.
\smallskip

(iii) We triangulate the prism $ABCA_1B_1C_1$ using the edges
$AC_1, BA_1, BC_1$. The prism $ADCA_1D_1C_1$ is triangulated by translating $BA_1$
and $BC_1$ onto the opposite facets.

In both cases, (ii) and (iii), the triangulations induced on opposite facets
of $Q_3$ agree under translation. The resulting triangulation therefore extends
periodically to $\Bbb R^3$ and has vertex set $\Bbb Z^3$. Moreover, the difference
of the endpoints of every edge is, up to sign, the difference of two of the five
selected vertices. Hence the covering follows from
Lemmas~\ref{lem:5} and~\ref{lem:4}.
\end{proof}

\begin{proof}[Alternative proof of Corollary~\ref{cor:3}]
Let $p$ be the apex of the considered pyramid.
Write the vertices of its base as $a+U\varepsilon$, $\varepsilon\in\{0,1\}^{n-1}$,
where $U=(u_1,\dots,u_{n-1})$ is an $n\times(n-1)$ matrix. Then the volume of the
pyramid is $|\det(U,p-a)|/n$. Thus it is enough to prove
\begin{equation}\label{eq:1}
\left|\det(U,p-a)\right|\le 1.
\end{equation}

Let $r_j=(u_{j1},\dots,u_{j,n-1})$ be the $j$-th row of $U$, and put
$$P_j=\sum_i\max(u_{ji},0),\quad N_j=\sum_i\max(-u_{ji},0).$$
Since every vertex $a+U\varepsilon$ lies in $Q_n$, the range of the $j$-th
coordinate has length $P_j+N_j$. Hence
\begin{equation}\label{eq:2}
P_j+N_j\le 1.
\end{equation}
Moreover,
\begin{equation}\label{eq:3}
N_j\le a_j\le 1-P_j.
\end{equation}

Let $c_j$ be the cofactor of the $j$-th entry in the last column of $(U,v)$, so
that
$$\det(U,v)=\sum_{j=1}^n c_jv_j.$$
By reflecting suitable coordinates of $Q_n$ and, if necessary, changing the
common sign of the cofactors, we may assume $c_j\ge0$ for all $j$. From
$0\le p_j\le1$ and \eqref{eq:3},
$$\sum_j c_j(p_j-a_j)\le \sum_j c_j(1-N_j),$$
and
$$-\sum_j c_j(p_j-a_j)\le \sum_j c_j(1-P_j).$$
It therefore remains to show that both sums on the right are at most $1$.

Consider first $\sum_jc_j(1-N_j)$. Form the $n\times n$ matrix $Y$ whose $j$-th
row is $y_j=(r_j,1-N_j)$. Expansion in the last column gives
$$\det Y=\sum_jc_j(1-N_j).$$
By \eqref{eq:2}, each row has the convex decomposition
$$y_j=(1-P_j-N_j)e_n +\sum_{u_{ji}>0}u_{ji}(e_n+e_i)+\sum_{u_{ji}<0}(-u_{ji})(-e_i).$$
After multiplying the last coordinate by $-1$, the vectors appearing in this
decomposition become reduced node--arc incidence vectors of a directed graph. The
node--arc incidence matrix of a directed graph is totally unimodular; see, e.g.,
\cite[p.~274]{S}. Hence every square matrix formed from these vectors has
determinant $0$ or $\pm1$. Since each $y_j$ is a convex combination of these
vectors, multilinearity of the determinant gives $|\det Y|\le1$. Therefore
$$\sum_jc_j(1-N_j)\le1.$$
Replacing $U$ by $-U$ interchanges $P_j$ and $N_j$ and changes all cofactors
by the same sign. The same argument therefore gives
$$\sum_jc_j(1-P_j)\le1.$$
This proves \eqref{eq:1}.
\end{proof}

\end{document}